\documentclass{amsart}
\usepackage{mathtools}
\usepackage{amssymb}
\usepackage{amsmath}
\usepackage{xcolor}
\usepackage{amsthm}
\usepackage{amsfonts}
\usepackage[utf8]{inputenc}
\usepackage{bm}
\usepackage{cite}
\usepackage[margin=1in]{geometry}

\newtheorem{thm}{Theorem}[section]

\newtheorem{cor}[thm]{Corollary}

\newtheorem{lemma}[thm]{Lemma}

\newtheorem{prop}[thm]{Proposition}

\newtheorem{defn}[thm]{Definition}
\newcommand\newnumbered[2]{\newtheorem{#1}{#2}}
\newnumbered{rem}{Remark}
\newnumbered{conclusion}{Conclusion}
\newnumbered{remarks}{Remarks}
\newnumbered{ex}{Example}
\newnumbered{exc}{Exercise}
\newnumbered{facts}{Facts}
\newnumbered{prob}{Problem}
\newnumbered{question}{Question}
\newnumbered{answer}{Answer}
\newnumbered{conj}{Conjecture}

\author{Bernhard G. Bodmann and Ahmed Abouserie}
\date{}
\begin{document}

\title[Spikes, Roots, and Modulations]{Spikes, Roots, and Modulations: Phase Retrieval\\ for Finitely-Supported Complex Measures}
\maketitle

\begin{abstract}
We study the recovery of a finitely supported distribution, a complex linear combination of Dirac measures, from intensity measurements.
The distribution $\mu=\sum_{j=1}^{s}c_{j}\delta_{t_{j}}$ is given by 
a coefficient vector $c\in\mathbb{C}^s$ and
its support $\{t_1, t_2, \dots, t_s\}$ is  contained in $ [0,\Lambda]$ for some $\Lambda>0$. The intensity measurements
evaluate (squared) magnitudes of a set of linear functionals applied to $\mu$, obtained by sampling $\hat \mu$, the Fourier transform of $\mu$, or by evaluating differences between modulated samples. 
Following a strategy by Alexeev et al., the structure of the linear functionals, and hence of the non-linear magnitude measurement, is encoded with a graph, where the vertices represent point evaluations of $\hat \mu$ at $\{v_1, v_2, \dots, v_n\} \subset [-\Omega,\Omega]$ and each edge represents a (modulated) difference
between vertices incident with it.
We show that a Ramanujan graph with degree $d \ge 3$ and 
 $n>\frac{6(1 + 6 /\ln(s/\Lambda\Omega)) s}{1-2\sqrt{d-1}/d}$ vertices provides $M=(d+1)n$
 magnitudes that are sufficient for identifying the complex measure up to an overall unimodular multiplicative constant.
 At the cost of including an additional oversampling step  and with an additional requirement that $n-1$ is prime, we  
construct an explicit recovery algorithm that is based on the Prony method.
\end{abstract}

\section{Introduction} Phase retrieval is an inverse problem with relevance in many fields such as crystallography \cite{DBLP:journals/corr/RanieriCLV13}, medical imaging \cite{yan_wu_liu_2013}, and quantum mechanics \cite{flammia2012quantum,Gross:2013fk,gross17improved}. The problem
is more concretely described as signal recovery when  only  the magnitude of linear measurements are known. 
The real case has been studied extensively and is fairly well understood,
in particular, if the signal of interest is  a vector in a finite dimensional Euclidean space \cite{balan2006signal,alexeev_bandeira_fickus_mixon_2014}.
Necessary conditions for the recovery of the signal were also provided in the complex case\cite{alexeev_bandeira_fickus_mixon_2014},
next to establishing a class of successful generic measurements \cite{ConcaEdidinEtAl:2015}. As demonstrated by Vinzant \cite{vinzant2015}, the smallest number of measured quantities needed for complex phase retrieval is still  an open problem. 

Apart from the characterization of measurement strategies that determine each signal, noise resilience and 
concrete recovery algorithms have been discussed, 
see \cite{fienup_1982} for an overview of algorithms applied in solving the discrete case. 
In the complex case, PhaseLift uses an embedding of the signal in a space of positive semidefinite matrices to formulate
feasible recovery methods for randomized measurements \cite{Candes:2012fk, Candes:uq}. 
An alternative method uses an embedding of complex polynomial spaces in a corresponding space of trigonometric polynomials
\cite{bodmann_hammen_2015,bodmann_hammen_2017}. 
The case of signals in infinite dimensional Hilbert spaces has been shown to provide no possible stable recovery \cite{cahill2016phase}.
However, one may hope that additional support restrictions and sparsity assumptions for the signal may help.

Indeed, in recent years there has been increasing interest in studying the continuous case and more specifically the case of complex-valued distributions. The latter model provides a more realistic description of physical systems, such as the case of X-ray crystallography where we are interested in determining the molecular structure of given samples \cite{DBLP:journals/corr/RanieriCLV13}.
Even when the measurement includes phase information, this is a challenging problem that has also been studied under the name of superresolution
\cite{Donoho92,CandesFernandezGranda2014}.

In \cite{beinert_plonka_2017}, the authors apply the Prony method to this problem to provide a sufficient condition for the recovery 
of a complex linear combination of $s$ Dirac measures,
$$
   \mu = \sum_{j=1}^s c_j \delta_{t_j}
$$
using a number of measured quantities that is $O(s^{2})$. Here, $\mu$ is determined up to an overall unimodular multiplicative constant,
 so instead of the signal, we get $[\mu]= \{c \mu: c \in \mathbb{C}, |c|=1\}$.
Uniqueness conditions exploiting the sparsity assumption were discussed in \cite{DBLP:journals/corr/RanieriCLV13}, where a certain condition (no collision) on these pairwise distances between nodes is shown to be sufficient for recovery of almost every signal of the above form. The key feature in both methods is that we attempt to recover the Fourier transform of the autocorrelation function, either by finding an inverse to a Hankel matrix in the first case, or by solving a system of quadratic equations in the second. 
We observe that the actual number of unknowns needed to determine the signal is only twice the support size and not quadratic in $s$, which motivates the results presented 
here. We wish to determine $\mu$ with a number of measured quantities that is linear in $s$ and provide a concrete recovery method.

This paper has two main contributions:
\begin{itemize}
\item  {\bf Injectivity of  measurements.}
We use the connectivity properties of a graph whose vertices are concrete points $V=\{v_1, v_2, \dots, v_n\}$ 
on the real line together with results from
complex analysis to establish that a measurement of sufficiently many magnitudes of the Fourier transform at these vertices, together with magnitudes of differences after applying a modulation to the Fourier transform, is sufficient to characterize the signal, up to a remaining multiplicative unknown unimodular constant. More precisely,
the first main contribution is formulated as follows:

Let $\Omega, \Lambda >0$, $\Omega \Lambda<1/2$, and $\mu$ be a complex linear combination
of $s$ Dirac measures with support in $[0,\Lambda]$. There is a constant $C>0$ depending on $\Omega \Lambda$ such that
if $d\geq 3$ 
and  $\Gamma$ is a Ramanujan graph with regularity $d$ and at least 
$n > \frac{Cds}{d - 2 \sqrt{d-1}}$ vertices $V=\{v_1, v_2, \dots, v_n\}$ in the set $[-\Omega,\Omega]$, then
there are  $M=(d+1)n$ magnitude measurements associated with the graph
that  determine  $[\mu]$.
For small $\Omega$ and $\Lambda$, the numerical constant $C$ can be chosen close to $6$. The overall number of measured quantities for $d=3$ is then close to $70s$, see Theorem~\ref{thm:main1} for details.
\item {\bf A recovery algorithm.}
The Prony method and an oversampling strategy provides a recovery algorithm that reduces to applying linear inverses and
finding roots of a polynomial.
More precisely, if $n> \frac{Cds}{d - 2 \sqrt{d-1}}$ and $n$ is prime, $d\ge 3$,
and $\Gamma$ is a Ramanujan graph of $n+1$ vertices, then there is a sampling strategy
associated with $\Gamma$ and an iterative phase-propagation algorithm resulting in
 $M=(d+1)(n+1)$ measured quantities that
provides values $\zeta \hat \mu(2k\Omega/(n+1))$ for each $k \in \{1, 2, \dots, (n+1)/2\}$ 
with a remaining residual unknown $\zeta \in \mathbb C$, $|\zeta|=1$. Subsequently applying the Prony algorithm, 
factoring a polynomial and solving a linear system then gives $[\mu]$.
\end{itemize}

The remainder of this paper is organized as follows: In Section~\ref{sec:prelim}, we introduce preliminary concepts and fix notation.
The section on main results is split into the part on injectivity, Section~\ref{sec:inject} and the last part on algorithms, 
Section~\ref{sec:algo}.

\section{Preliminaries}\label{sec:prelim}

The class of signals we study are certain complex linear combinations of Dirac measures on the real line. The simplest case, the (normalized) Dirac measure supported at $t \in \mathbb R$ is denoted by $\delta_t$. We restrict the support to $t \in [0,\Lambda]$.
More generally, $S_{[0,\Lambda],s}$ is the set of such linear combinations having support of maximal size $s \in \mathbb{N}$ contained in the interval $[0,\Lambda]$, formally defined as 
\begin{align*}S_{[0,\Lambda],s}=\Bigl\{&\mu=\sum_{j=1}^{s}c_{j}\delta_{t_{j}}:\, c \in \mathbb{C}^s
\text{ and } 0\leq t_{1}<t_{2}<...<t_{s}\leq\Lambda\Bigr\} \, .
\end{align*}
Let the union of all such sets be $S_{[0,\Lambda]} = \cup_{s=1}^\infty S_{[0,\Lambda],s}$,
and for any measure $\mu = \sum_{j=1}^{s}c_{j}\delta_{t_{j}} \in S_{[0,\Lambda]}$, let  the Fourier transform $\mathcal{F}(\mu)=\hat{\mu}:\mathbb{R}\rightarrow\mathbb{C}$  
be chosen according to the convention so that  its value at $\omega \in \mathbb R$ is
\begin{equation*}\hat{\mu}(\omega)
=\sum_{j=1}^{s}c_{j}e^{-2\pi i\omega t_{j}}.\end{equation*} 

In pursuit of injectivity, we observe that for any linear functional $\phi$ on $S_{[0,\Lambda]}$ and any $c\in \mathbb{C}$ with $|c|=1$, 
we have that $|\phi(c\mu)|^{2}=|\phi(\mu)|^{2}$. Since our measurements are given by such squared magnitude evaluations, our procedure will aim at recovering the equivalence class $[{\mu}]\in S_{[0,\Lambda]}/\sim$\, where $\sim$ is the equivalence relation defined by $\mu \sim \nu$ if and only if $\mu = u \nu$, with $u \in \mathbb C$, $|u|=1$.

The point evaluations of $\hat \mu$ are part of the linear functionals we include in the measurement. However, if
for some $\omega \in \mathbb R$, $\hat \mu(\omega)=0$, then only limited information can be deduced from this.
Fortunately,  we have the following upper bound on the number of roots that $\hat{\mu}$ has in $[-\Omega,\Omega]$.

\begin{lemma}[Theorem 3 in \cite{tijdeman_1971}] Let $\Lambda,\Omega>0$, $s \in \mathbb N$ with $s>\Lambda\Omega$, and $\mu\in S_{[0,\Lambda]}$ having a support of size at most $s$, 
then if $\mu \ne 0$, the number of roots of the Fourier transform $\hat \mu$
in $[-\Omega,\Omega]$ cannot exceed $(1 + 6 /\ln(s/\Lambda\Omega)) s$.
\end{lemma}

We formulate a simple consequence for signal recovery from (linear) sampling of the Fourier transform $\hat \mu$. 
Let the sampling set be $\{v_1, v_2, \dots, v_n\} \subset [-\Omega,\Omega]$,
and evaluate the Fourier transform of $\mu$ at these points, so the corresponding linear map is
$\phi_v: S_{[0,\Lambda]} \to \mathbb{C}^n$, $(\phi_v\mu)_j = \hat \mu(v_j)$.\\

\begin{prop}
Let $\Lambda,\Omega>0$, $s \in \mathbb N$, 
$n > (2 + 12 /\ln(2s/\Lambda\Omega)) s$,
then the sampling map $\phi_v$ associated with
$\{v_1, v_2, \dots, v_n\} \subset [-\Omega,\Omega]$,
maps $S_{[0,\Lambda],s}$ injectively to $\mathbb{C}^n$.
\end{prop}
\begin{proof} If
$\mu_1, \mu_2 \in S_{[0,\Lambda],s}$ give $\phi_v\mu_1 = \phi_v\mu_2$,
then $\phi_v(\mu_1-\mu_2)=0$. Since $\mu=\mu_1 - \mu_2 \in  S_{[0,\Lambda],2s}$,
$\hat\mu$ has $n$ zeroes. This is only possible if $\mu=0$, so $\mu_1=\mu_2$.
\end{proof}

\section{Main Results}

Next, we wish to consider nonlinear, squared magnitude measurements. For this purpose, we consider the finite set of points $V=\{v_1, v_2, \dots, v_n\}$ in $[-\Omega,\Omega]$,  and identify these points
with the vertices in a graph $\Gamma$ having an edge set $E$. 
Typically, graph theory does not require a specific choice of the set labeling the vertices. Here, we apply the properties of graphs
to establish a desirable measurement structure and for this purpose choose the vertices to be a {\it concrete} subset of $\mathbb R$.  

Given $\mu$ as described, we
consider the following data consisting of intensity values
\begin{align*}
&\mathcal{M}_{0,\Gamma}(\hat\mu) =(|\hat{\mu}(v_{j})|^{2})_{j=1}^{n},\\
&\mathcal{M}_{1,\Gamma}(\hat\mu)=(|\hat{\mu}(v_{j})-\hat{\mu}(v_{k})|^{2})_{\{v_j,v_k\} \in E,j<k},\\ \text{and}
\quad&\mathcal{M}_{2,\Gamma}(\hat\mu) =(|\hat{\mu}(v_{j})-i\hat{\mu}(v_{k})|^{2})_{\{v_j,v_k\}\in E,j<k}.\end{align*}

This implies that the total number of measured quantities is the sum $M= |V|+2|E|$.

\subsection{Injectivity}\label{sec:inject}

The first goal is to derive conditions for injectivity of the measurement.
We rely on the following elementary lemma, which is at the heart of the so-called phase propagation algorithm used in \cite{bodmann_hammen_2017}, see also \cite{Balan:2009fk} and \cite{bodmann_hammen_2015}.

\begin{lemma}
If the entries of the column vectors $a,b \in \mathbb C^2$ satisfy $|a_1|=|b_1|$, $|a_2|=|b_2|$
as well as $|a_1-a_2| = |b_1 - b_2|$ and $|a_1-i a_2| = |b_1 -i b_2|$,
then the matrix identity $a a^* = b b^*$ holds. Moreover, if $a_2 \ne 0$, then
$a_1/a_2 = b_1/b_2$.
\end{lemma}
\begin{proof}
 The stated identities can be re-expressed as $\mathrm{tr}[ A_i aa^* ] = \mathrm{tr}[ A_i b b^*]$
with $i \in \{1,2,3,4\}$ and $A_1 = e_1 e_1^*$, $A_2=e_2e_2^*$, $A_3 = e_1 e_1^*-e_1 e_2^*-e_2 e_1^* +e_2 e_2^*$
and $A_4 =  e_1 e_1^*- ie_1 e_2^*+ i e_2 e_1^* +e_2 e_2^*$, where $e_i$ is the $i$th canonical basis vector. 
Since $\{A_1, A_2, A_3, A_4\}$ is a vector-space basis for the space of all $2 \times 2$ matrices,  the rank-one matrices obtained from $a$ and $b$ are identical, $a a^* = b b^*$.
This implies $a_1 \overline{a_2} = b_1 \overline{b_2}$. Assuming $a_2 \ne 0$ and using $|a_2|^2 = |b_2|^2$, we get
the desired identity for the quotient.
\end{proof}

As a simple consequence, if  $\mu \in S_{[0,\Lambda],1}$ then the recovery of $[\mu]$ is possible with a measurement
based on the graph with one edge and two vertices.


\begin{prop}
Let $\mu \in S_{[0,\Lambda],1}$, so $\mu = c \delta_t$ with $c \in \mathbb C$ and $t \in [0,\Lambda]$, and $\{v_1, v_2\} \subset [-\Omega,\Omega]$ with $\Omega\Lambda< 1/2$, 
then the graph $\Gamma$ with vertices $V=\{v_1, v_2\}$ and one edge $E=\{\{v_1,v_2\}\}$ permits the recovery of 
$[\mu]$ from $\{\mathcal{M}_{i,\Gamma}(\hat\mu)\}_{i=0}^2$ with a total of 4 measured magnitudes.
\end{prop}
\begin{proof}
Equivalently to the recovery of $[\mu]$, we have to determine the unknown parameter $t \in [0,\Lambda]$ in $\hat \mu (\omega)= c e^{-2\pi i t \omega}$, and the magnitude $|c|$ of the factor $c \in \mathbb C$.
We know $\mu=0$ if and only if $\hat \mu(v_1)= \hat \mu(v_2)=0$, so it
remains to treat the case when $\hat \mu$ does not vanish at $v_1$ or $v_2$, but then it is non-zero at both points.
In fact, we have $|c|=|\mu(v_1)|$, so it only remains to identify $t$.

According to the preceding lemma, measuring the interference magnitudes
together with the intensities at $v_1$ and $v_2$ permits us to compute the quotient $\hat \mu (v_1)/\hat \mu(v_2) = e^{-2\pi i t (v_1-v_2)}$,
and from $0 \le t|v_1-v_2| \le 2 \Lambda \Omega<1$, we can solve for $t$.
\end{proof}

Computing $\hat\mu(v_i)/\hat \mu(v_j)$ from the measured data is possible if $\hat \mu(v_j) \ne 0$. We then say the ``phase propagates from $v_j$ to $v_i$ along the edge $\{v_i,v_j\}$''. If there is a path in $\Gamma$ such that each vertex $v_j$ along the path corresponds to a  non-vanishing value $\hat \mu (v_j)$,
then fixing the phase at one end determines the values of $\hat \mu$ at all other vertices of the path.
If $\Gamma$ has a spanning tree where the vanishing values of $\hat \mu$ are in a subset of the leaves, then fixing the phase at the root permits to propagate the phase to the entire graph. In general, this may not be possible because of vertices where $\hat \mu$ vanishes, but then it is still possible to propagate the phase among vertices in a subset.

\begin{defn} Let $\Gamma=(V,E)$ be a graph. The subgraph of $\Gamma$ induced by a subset  $W \subset V$ is given by
the vertex set $W$ and the set of edges  $\{v_j,v_k\} \in E $ with $\{v_j, v_k \} \subset W$. Given a complex measure $\mu$ of finite support, we often 
choose $W=\{v \in V: \hat \mu(v) \ne 0\}$ and consider the subgraph $\Gamma_\mu$ induced by $W$.
\end{defn}

In order to ensure an induced subgraph with a large connected component, 
we exploit properties of expander graphs, as presented in \cite{alexeev_bandeira_fickus_mixon_2014}. As a first step, we
formulate a general result.

\begin{thm}\label{thm:lin_inject}
Let $s \in \mathbb N$, $\Omega, \Lambda >0$ with $\Omega \Lambda < 1/2$, $\mu \in S_{[0,\Lambda],s}$ and $\Gamma=(V,E)$ be a graph 
such that the induced subgraph $\Gamma_\mu$ has
at least one connected component with 
 $(2 + 12 /\ln(2s/\Lambda\Omega)) s$ vertices, then the measurement $\{\mathcal{M}_{i,\Gamma}(\hat\mu)\}_{i=0}^2$
 determines $[\mu]$.
\end{thm}
\begin{proof}
We note that phase propagation permits us to assign values to $\{\hat \mu(v): v \in K\}$ where $K$ is a connected component of $W$
so that these values are consistent with the measurements. In more detail, after fixing a representative $\nu$ of $[\mu]$ by choosing $v_0 \in K$
and letting $\hat \nu(v_0) = |\hat \mu(v_0)|$, the phase propagation algorithm assigns values to its neighbors, and successively to 
the entire connected component $K$. The measurement with the subgraph of $\Gamma$ induced by $K$ thus provides $\hat \nu |_K$.  
If $K$ has the size stated in the assumption, then by the injectivity result, this 
is sufficient to characterize $[\mu]$.
\end{proof}

The implicit condition involving the signal-dependent subgraph $\Gamma_\mu$ may be acceptable for random signals
for which vanishing measurements occur with probability zero. In order to provide deterministic 
performance guarantees, Alexeev at al.\ \cite{alexeev_bandeira_fickus_mixon_2014} suggested the use of expander graphs, which are useful for
 achieving large connected components.
 

%
%

\begin{defn}  Let $\Gamma=(V,E)$ be a finite graph. Denote by $\mathcal{E}(W)$ the set of edges having one vertex in $W\subset V$ and the other in the complement of $W$ in $V$. Formally, 
\begin{equation*}\mathcal{E}(W)\vcentcolon=\{e \in E,  e \cap W\neq \emptyset \,\text{and}\, e \cap W^{c}\neq \emptyset\}.\end{equation*}
 The \textit{expansion constant} $h(\Gamma)$ is then defined by 
 \begin{equation*}h(\Gamma)\vcentcolon=\inf\Bigl\{\frac{|\mathcal{E}(W)|}{|W|}\in[0,\infty)\,\big|\,\emptyset\neq W\subset V \,\text{and}\,|W|\leq\frac{1}{2}|V|\Bigr\}\end{equation*}
such that $h(\Gamma)=\infty$ if $\Gamma$ has at most one vertex.

\end{defn}

\noindent\textit{Remark 1.5.} The expansion constant is a measure of the robustness of the graph in the sense that large values of $h(\Gamma)$ correspond to more difficulty in disconnecting large subsets of $V$ from the rest of the graph.\\

%

\begin{defn} Let $\Gamma=(V,E)$ be a finite graph. The \textit{normalized Laplace operator} of $\Gamma$, $L_{\Gamma}$, is the 
self-adjoint
linear operator \begin{align*}L_{\Gamma}:\begin{cases}
     \ell^{2}(V)\rightarrow \ell^2(V) \\
     L_\Gamma\varphi(v) = \varphi(v) - \frac{1}{|\mathcal{E}(\{v\})|}\sum_{\{v,w\} \in E}\varphi(w)
    \end{cases}\end{align*}\end{defn}

\noindent\textit{Remark.} It can be easily checked that $L_\Gamma$ is positive semi-definite and the function $\varphi:V\rightarrow\mathbb{C}$, defined by $\varphi(x)=1$ for all $x\in V$, is an eigenfunction of $L_{\Gamma}$ corresponding to the eigenvalue $\lambda_{0}=0$. 
The \textit{normalized spectral gap}, $\lambda_{1}(\Gamma)$, is the smallest nonzero eigenvalue of $L_{\Gamma}$.


Following the general strategy outlined by Alexeev et al.\ \cite{alexeev_bandeira_fickus_mixon_2014}, we choose 
a measurement based on a Ramanujan graph $\Gamma$. The size of $\Gamma$ is chosen such that after removing at most $(1 + 6 /\ln(s/\Lambda\Omega)) s $ vertices and edges adjacent to them, we may invoke properties of expander graphs to show that there remains a connected component that is large enough to guarantee that phase propagation returns a unique solution, up to the residual undetermined unimodular factor. 

We recall the relationship between $\lambda_1(L_\Gamma)$ and the expansion constant of the graph $\Gamma$.
This relationship  is of importance, as it is easier to control $\lambda_{1}(L_\Gamma)$ than  $h(\Gamma)$ directly.

\begin{prop}[Discrete Cheeger Inequality for $d$-regular graphs, Proposition 3.1.2\cite{kowalski_2018}] 
Let $\Gamma$ be a connected, non-empty, finite $d$-regular graph, then
$\lambda_{1}(L_\Gamma)\leq\frac{2}{d}h(\Gamma).$
\end{prop}

\begin{lemma}[Lemma 5.2 in \cite{harsha_barth_2005}] Let $d\geq 2$. Let $\Gamma=(V,E)$ be a $d$-regular connected graph. For all $\delta< \frac{\lambda_{1}(\Gamma)}{12}$, removing at most $2\delta d |V|$ edges from the graph will leave a connected component of size at least $2|V|/3$.
\end{lemma}
\begin{proof}
If removing the edges divides the graph into at least two connected components, then there is $W\subset V$ with $|W|\leq\frac{1}{2}|V|$ and $|\mathcal{E}(W)|\leq2\delta d |V|$. This implies by Cheeger's inequality that \begin{align*}|W|\lambda_{1}(\Gamma)\frac{d}{2}&\leq|W| h(\Gamma)\\
&\leq|\mathcal{E}(W)|\\
&\leq 2\delta d |V|\\
&< 2\frac{\lambda_{1}(\Gamma)}{12} d |V|,\end{align*} which gives \begin{equation*}|W|<\frac{1}{3}|V|.\end{equation*} 
This shows that every subset $W$ that gets separated from the rest of the graph and is of size $|W|\leq\frac{1}{2}|V|$, must have a size 
less than $\frac{1}{3}|V|$. By taking complements, we obtain that if $W^c=V\setminus W$ becomes a connected component
 of size $|W^c| \ge |V|/2$ after the edges are removed,
then it actually has the larger size $|W^c| > 2 |V|/3$. To finish the proof, we show the existence of such a connected component.

Assume each connected component is of size at most $|V|/2$, then their individual size is less than $|V|/3$. 
Moreover, taking the union of two disjoint sets $W_1$ and $W_2$, each of size at most $|V|/2$ and disconnected from their complements,  
then $|W_1 \cup W_2 | = |W_1 | + |W_2| < 2 |V|/3$. If $|W_1  \cup W_2| \ge |V|/2$, then by taking complements
and using $|(W_1 \cup W_2)^c| < |V|/3$ we get
$|W_1 \cup W_2| > 2 |V|/3$, a contradiction. This implies $|W_1 \cup W_2 | \le |V|/2$, and in turn, $|W_1 \cup W_2 | < |V|/3$.
After taking the union of all connected components, we still have a set of size less than $|V|/3$, which contradicts that the 
graph has $|V|$ vertices. Hence, there exists a connected component $W$ with $|W|>|V|/2$, and thus
$|W| > 2|V|/3$.
\end{proof}


 Ramanujan graphs are known to give the best asymptotic value for  $\lambda_{1}(L_\Gamma)$. We use them to provide 
 expansion properties.


\begin{defn} Let $d\geq 3$ be an integer. A $d$-regular connected finite graph $\Gamma$ is
called \textit{a Ramanujan graph} if $\lambda_1(L_\Gamma)\ge 1-\frac{2\sqrt{d-1}}{d}$.
\end{defn}

\noindent\textit{Remark 1.16.} Proving the existence of expander families in general, and expander families of Ramanujan graphs in particular, has been challenging. We refer to \cite[Theorem 1.1]{marcus_spielman_srivastava_2015} for the proof of the existence of the latter for each $d\ge 3$ and vertex sets that are of even size.
The simplest construction presented there gives bipartite Ramanujan graphs with $n=2^{r+1} d$ vertices, for any $r \in \mathbb N$.



\begin{prop} Let $d \ge 3$ and $\Gamma=(V,E)$ be a $d$-regular Ramanujan graph, then
removing less than $\Bigl(\frac{d}{6} - \frac{\sqrt{d-1}}{3} \Bigr) |V|$ edges from $\Gamma$ leaves a connected
component with at least $2|V|/3$ vertices. 
\end{prop}
\begin{proof}
The  assumption $2\delta d |V| < \Bigl(\frac{d}{6} - \frac{\sqrt{d-1}}{3} \Bigr) |V|$
implies $\delta < \frac{1}{12} - \frac{\sqrt{d-1}}{6d}$ and by the inequality for $\lambda_1(L_\Gamma)$ of the Ramanujan graph,
$\delta < \frac{\lambda_1(\Gamma)}{12}$, so the previous lemma applies.
\end{proof}



We conclude a condition that ensures injectivity of measurements. For this purpose, we rely on the connectivity and spectral properties of Ramanujan graphs.

\begin{thm} \label{thm:main1} Let $\mu\in S_{[0,\Lambda]}$ be a complex measure of support $s \in \mathbb N$. 
Let $d\geq 3$ and let $\Gamma = (V,E)$ be a Ramanujan graph with regularity $d$ and at least 
$n > \frac{6d(1+6/\ln(s/\Lambda\Omega))s}{d - 2 \sqrt{d-1}}$ vertices $V=\{v_1, v_2, \dots, v_n\}$ in the set $[-\Omega,\Omega]$ with $2 \Omega \Lambda<1$. 
The  $M=(d+1)n$ magnitude measurements $\{\mathcal{M}_{i,\Gamma}(\hat\mu)\}_{i=0}^2$ associated with the graph $\Gamma$ then determine  $[\mu]\in S_{[0,\Lambda]}/\sim$.
\end{thm}

\begin{proof}
Let $\Gamma=(V,E)$ be a Ramanujan graph  with the claimed regularity and size, whose vertices
are chosen to be points $\{v_1, v_2, \dots, v_n\} \subset [-\Omega,\Omega]$.
We first note that $\hat \mu$ has at most $(1+6/\ln(s/\Lambda\Omega))s$ zeros in $[-\Omega,\Omega]$.
This means, when considering the induced subgraph $\Gamma_\mu$ for which only vertices are included where $\hat \mu$
does not vanish, then at most  $d (1+6/\ln(s/\Lambda\Omega))s$ edges have been removed from $\Gamma$.

With the assumption on the size of the graph $\Gamma$, $d (1+6/\ln(s/\Lambda\Omega))s \le (d- 2 \sqrt{d-1}) |V|/6$. 
By the expander property of this graph, $\Gamma_\mu$ has a connected component whose vertex set $K$ has a size at least
$$
   |K| \ge \frac 2 3 |V| = \frac{4(1+6/\ln(s/\Lambda\Omega))s}{1 - 2 \sqrt{d-1}/d}
   > 2 (1+6/\ln(s/\Lambda\Omega))s \, .
$$
Using the $M=(d+1) |V|$ measured quantities from $\{\mathcal{M}_{i,\Gamma}(\hat \mu)\}_{i=0}^2$,
we  assign values $c \hat \mu(v_j)$ for each $v_j \in  K$ with some residual unknown $c \in \mathbb C, |c|=1$.
Finally, applying Theorem~\ref{thm:lin_inject}, we
observe that $[\mu]$ is determined by $c \hat \mu |_K$. 
\end{proof}

\subsection{Recovery algorithms}\label{sec:algo}

Our next result shows that it is possible to reconstruct a given signal $\mu\in S_{[0,\Lambda]}$ explicitly, using 
the combination of an oversampled measurement, phase propagation and the Prony method.


Let $\Gamma$ be a Ramanujan graph as described, with vertices $\{v_1, v_2, \dots, v_n\}$. We choose
$v_k =  k \Omega/n$, hence the vertices are spaced equidistantly with step-size $h=\Omega/n$.

We introduce simplifying notation and borrow ideas from time-frequency analysis. 
Let $z_j = e^{-2\pi i t_j h}$ and 
$M_{z_j}$
be a modulation operator on $\ell^2(\{1,2,\dots, n\})$ given by
$$
  M_{z_j} f(k) = z_j^k f(k) \, . 
$$
The function $\hat \mu$ restricted to $V$ can then be identified with a multiplication operator
on $\ell^2(\{1,2,\dots, n\})$, a linear combination of modulation operators
$$
    T = \sum_{j=1}^s {c_j} M_{z_j} \, 
$$
and the recovery algorithm aims to convert the implicit information obtained from the action of $T$
into extracting the values of each $t_j$ and $c_j$. It turns out, once the set $\{t_j\}_{j=1}^s$ is known,
the coefficient vector $c$ is straightforward to obtain.
To motivate this insight, we state a simple linear independence property of modulation operators.

\begin{lemma}
Let $\Lambda, \Omega>0$, $\Lambda \Omega < 1/2$, $n \in \mathbb N$, $\{t_1, t_2, \dots, t_s\}\subset [0,\Lambda]$, and for each $j \in \{1,2, \dots, s\}$, 
$z_j  = e^{-2\pi i t_j h}$ with $h=\Omega/n$.
If $n>(1 + 6 /\ln(s/\Lambda\Omega)) s$, then
the associated set of modulation operators $\{M_{z_j}\}_{j=1}^s$ on $\ell^2(\{1,2,\dots, n\})$
is a linearly independent set.
\end{lemma}
\begin{proof}
The linear independence property is a consequence of the fact that the diagonal entries of each modulation
operator is given by a complex exponential. If a linear combination of modulation operators vanishes, then all
diagonal entries of this linear combination vanish. By the result on the number of roots of linear combinations of complex exponentials,
if $n$ is sufficiently large as assumed, this is only possible if the linear combination is trivial.
\end{proof}

\begin{cor}
Under the assumptions of the preceding lemma, 
given the operator $T=\sum_{j=1}^s c_j M_{z_j}$ on $\ell^2(\{1,2,\dots, n\})$
and the values $\{z_1, z_2, \dots, z_s\}$, then the identity for $T$ has
a unique solution $c \in \mathbb{C}^s$. 
\end{cor}


The Prony method uses the fact that complex exponentials are eigenvectors of the shift operator.
We define the cyclic shift $S$ by $Sf(k) = f(k-1)$ and $Sf(1)=f(n)$. We recall a version of Weyl commutation relations.

\begin{lemma}
If $f \in \ell^2(\{1,2, \dots, n\})$ has support in $\{1,2, \dots, n-1\}$, then for any $z \in \mathbb C$, $|z|=1$,
$$
   S^* M_z S f = z M_z f \, .
$$
\end{lemma}
\begin{proof}
By linearity, it is enough to prove this for each canonical basis vector $e_k$ with $k \in \{1, 2, \dots, n-1\}$.
We note that by definition, $S e_k = e_{k+1}$ and $M_z e_k = z^k e_k$. Since $S^*$ is the inverse of $S$,
$S^* M_z S e_k = z^{k+1} e_k = z M_z e_k$.
\end{proof}

Together with linearity and iterating the application of $S$, the Weyl commutation relation can be used to generate an annihilating polynomial for $T$.
To this end, we associate a map $\Psi_q$ with each complex polynomial $q$ of degree $m$,
that maps any operator $X$ on $\ell^2(\{1,2, \dots, n\})$ to
$$
  \Psi_q(X) = \sum_{l=0}^m q_l (S^l)^* X S^l \, ,
$$
where $q(z) = \sum_{l=0}^m q_l z^l \, .$

\begin{lemma}
If $f \in \ell^2(\{1,2, \dots, n\}$ has support in $\{1,2, \dots, m\}$, then for any $z \in \mathbb C$, $|z|=1$,
and any complex polynomial $q$ of degree at most $n-m$,
$$
  \Psi_q(M_z) f=  \sum_{l=0}^{n-m} q_l (S^l)^* M_z (S^l) f = q(z) M_z f \, ,
$$
where $q(z) = \sum_{l=0}^{n-m} q_l z^l$.
\end{lemma}
\begin{proof}
For each monomial, this is true by iterating the commutation relation in the preceding lemma,
because $S^l e_k = e_{k+l}$ and if $k, l \le m$, then $k+l \le n$. The statement for polynomials follows
by linearity.
\end{proof}

\begin{thm}
Let $\Lambda, \Omega>0$, $n \in \mathbb N$, $\{t_1, t_2, \dots, t_s\}\subset [0,\Lambda]$, and for each $j \in \{1,2, \dots, s\}$, 
$z_j  = e^{-2\pi i t_j h}$ with $h=\Omega/n$.
Let $T= \sum_{j=1}^s c_j M_{z_j}$ on $\ell^2(\{1,2,\dots, n\})$ with 
$n\ge s+m$, $m>(1 + 6 /\ln(s/\Lambda\Omega)) s$ and
non-vanishing complex coefficients $\{c_1, c_2, \dots, c_s\}$, then there is a unique monic polynomial $q$ of degree $s$ such that
for each $f\in \ell^2(\{1,2, \dots, n\})$ with support in $\{1,2, \dots, m\}$, 
$$
   \Psi_q(T) f = \sum_{l=0}^s q_l (S^l)^* T S^l f  = 0 \, .
$$
\end{thm}
\begin{proof}
We note that by the lemma, $\sum_{l=0}^s q_l (S^l)^* T S^l f  = \sum_{j=1}^s q(z_j) c_j M_{z_j} f$.
If $q(z) = \prod_{j=1}^s (z-z_j)$, then indeed $q(z_j)=0$ and the statement is true. Moreover,
requiring the degree of $q$ to be $s$, yields only one choice of a monic polynomial that annihilates $T$. By the linear independence
of $\{M_{z_j}\}_{j=1}^s$ and $c_j \ne 0$, if $\sum_{j=1}^s q(z_j) c_j M_{z_j} f = 0$ for each $f$ with support of size $m$, then $q(z_j)=0$ for each
$j \in \{1, 2, \dots, s\}$. This forces $q$ to be the (monic) polynomial given by the factored form.
\end{proof}

The uniqueness of the annihilating monic polynomial can be formulated in the usual form of the Prony method.
The reason why we have used the operator formulation with modulations is its convenience for 
deriving the relation between $s$ and $n$.

\begin{cor}
Under the assumptions of the preceding theorem, the solution of the homogeneous linear system
given by the equations for each $k \in \{1,2,\dots m\}$
$$
   \sum_{l=0}^s q_l \sum_{j=1}^s c_j z_j^{k+l}  = 0
$$
satisfying $q_s=1$ gives the unique monic polynomial $q(z) = \prod_{j=1}^s (z-z_j)$.
\end{cor}
\begin{proof}
This follows from the definition of $T$ and from the choice $f=e_k$, $k \in \{1,2, \dots, m\}$.
\end{proof}

If $\hat\mu(v_k) $ is known for each $k \in \{1,2, \dots, m+s\}$, then the Prony method can be applied to extract 
$t_j$ and $c_j$. However, guaranteeing that phase propagation gives all the values seems to require
assumptions on $\Gamma_\mu$.
In order to avoid such implicit assumptions, we hereafter include an additional 
inverse Fourier transform and embed the graph in the signal domain so that missing data can be reconstructed
through oversampling.

Let 
$$
   \tilde \mu(t) = \sum_{k=1}^m \hat \mu(k\Omega/m) e^{i \pi t k/\Lambda} \, .
$$
We note that among all $2\Lambda$-periodic trigonometric polynomials of degree $m$, $\tilde \mu$ has $m+1$ vanishing coefficients, meaning we have implicitly set
$\hat \mu(k\Omega/m)=0$ if $k \in \{-m, -m+1, \dots, -1, 0\}$. 
To quantify the improvement from this inverse Fourier transform, we appeal to Terry Tao's  
strengthened uniform uncertainty principle.

\begin{lemma}[Corollary 1.4 in \cite{tao05}]
If $G=\mathbb{Z}_p$ and $p$  is prime, and $f \mapsto \hat f$ denotes the Fourier transform on $G$,
so $\hat f(\xi) = \frac{1}{p} \sum_{x \in G} f(x) e^{2 \pi i x \xi/p}$, then if $S, \tilde S \subset G$,
$F: \ell^2(S) \to \ell^2(\tilde S)$ is invertible, where $F f(\xi) = \hat f(\xi)$ and $\ell^2(S)$ denotes the functions on $G$ vanishing 
on $G\setminus S$.
\end{lemma}

As a consequence, we note that evaluating the function $\tilde \mu$ at any subset of (at least) $m$ vertices
from $\{v_k\}_{k=1}^{2m+1}$ with $v_k = 2 k \Lambda/(2m+1)$
determines $\{\hat \mu (k \Omega/m)\}_{k=1}^m$, if $2m+1$ is prime.

\begin{cor}
Let $\mu$ and $\tilde \mu$ be as described, then for any subset $\tilde S \subset \{1, 2, \dots, n\}$, $n=2m+1$ prime,
having size $|\tilde S| \ge m$,
the linear re-sampling operator $\Psi: (\hat\mu(k\Omega/m))_{k=1}^m \mapsto (\tilde \mu(2 k\Lambda/n))_{k \in \tilde S}$ is 
invertible.
\end{cor}

 \begin{cor}
Let $\Lambda, \Omega>0$, $\Omega>0$, $\Lambda \Omega < 1/2$ and 
$\mu$ be a complex measure having support $\{t_1, t_2, \dots , t_s\}\subset [0,\Lambda]$.
If $n> \frac{6d(1+6/\ln(s/\Lambda\Omega))s}{d - 2 \sqrt{d-1}}$ and $n$ is prime, $d\ge 3$ 
%
and $V=\{v_0,v_1, v_2, \dots, v_n\}$ 
with $v_k = h k$ and $h = 2 \Lambda/n$
for each $v_k$, 
then phase propagation applied to the oversampled measurement $\{\mathcal{M}_{i,\Gamma}(\tilde\mu)\}_{i=0}^2$
with a Ramanujan graph $\Gamma$ of degree $d$ having the vertex set $V$
provides values $\zeta\hat \mu(2k\Omega/(n+1))$ for each $k \in \{1, 2, \dots, (n+1)/2\}$ 
with a remaining residual unknown $\zeta \in \mathbb C$, $|\zeta|=1$. Moreover, 
the unique monic polynomial $q$ of degree $s$ satisfying the homogeneous system of
equations for $k \in \{1,2, \dots, s+1\}$ 
$$
   \sum_{l=0}^s q_l \hat \mu (2(k+l)\Omega/(n-1))   = 0
$$
factors into $q(z) = \prod_{j=1}^s (z-z_j)$
with $z_j = e^{-2\pi i t_j \Omega/n}$, and the system $\sum_{j=1}^s c_j z_j^k = c \hat \mu(v_k)$
with $k \in \{1, 2, \dots, n\}$ has a unique solution $c$.
\end{cor}
\begin{proof}
The assumptions for phase propagation and the Prony method need to be combined.
We note that if the number of vertices in the Ramanujan graph is $n+1$ with $n>\frac{6(1+6/\ln(s/\Lambda\Omega))s}{1 - 2 \sqrt{d-1}/d}$, then phase propagation
provides a consistent set of values $\{\zeta\tilde \mu(v_k): v_k \in K\}$ with a residual unknown factor $\zeta \in \mathbb C$,
$|\zeta|=1$. By the expander property, $|K|> 2|V|/3 = 2(n+1)/3$.
If $n$ is prime, using the fact that the linear map from $(\hat \mu(k \Omega/m)_{k=1}^m$ with $m=(n+1)/2$ to $(\tilde\mu(v_k))_{v_k \in K}$
is invertible, we obtain sample values for $\hat \mu$. Next,
 $$(n+1)/2> \frac{3(1+6/\ln(s/\Lambda\Omega))s}{1-2 \sqrt{d-1}/d} > 3(1+6/\ln(s/\Lambda\Omega))s$$ hence $(n+1)/2>m+s$ with 
 $m=(1 + 6 /\ln(s/\Lambda\Omega)) s$.
This implies that the Prony method can be applied to find the polynomial whose roots are $\{e^{-2\pi i t_j}\}_{j=1}^s$.
Knowing the support of $\mu$ then allows to solve for the coefficient vector $c$.
\end{proof}

As mentioned previously, a family of expander graphs can be obtained with the size $2^{r+1} d$. If we choose $d=4$,
 then there are examples of $ n-1 = 2^{r+3} -1$ being (Mersenne) primes. Fortunately, there are constructions for more
 general sizes of graphs as well \cite{marcusetal_2015}, see also \cite{MORGENSTERN199444,Cohen2016}.

%


%

\end{document}